\documentclass[leqno,12pt]{amsart}
\usepackage{url}
\usepackage{latexsym}
\usepackage{amssymb,amsmath,amsfonts}
\usepackage{mathtools}
\usepackage{array}
\usepackage[utf8]{inputenc}
\usepackage[usenames,dvipsnames]{color}

\newtheorem{theorem}{Theorem}[section]

\newtheorem{lemma}[theorem]{Lemma}
\newtheorem{proposition}[theorem]{Proposition}
\newtheorem{corollary}[theorem]{Corollary}
\theoremstyle{definition}
\newtheorem{definition}[theorem]{Definition}
\newtheorem{example}[theorem]{Example}

\newtheorem{remark}[theorem]{Remark}

\numberwithin{equation}{section}
\usepackage[left=2.5cm,right=2.5cm,top=1.8cm,bottom=1.8cm]{geometry}

\begin{document}

\title[Countable hybrid methods with monotone inclusions]{A hybrid method for countable equilibrium, variational inequality and maximal monotone inclusion problems with fixed point constraints}

\subjclass[2020]{47H09, 47H05, 47J25, 47J05, 49J40, 65K15}
\keywords{Equilibrium problem, variational inequality problem, maximal monotone operator, generalized $J_{*}$-nonexpansive mapping, generalized projection, hybrid method, strong convergence}

\maketitle
\begin{center}
 Markjoe O. Uba\\
 \vspace{.1in}
Department of Mathematical Sciences,\\ 
Northern Illinois University,\\ 
DeKalb, IL 60115, USA\\
{\tt markjoeuba@gmail.com}  
\end{center}

\begin{abstract}
Let $C$ be a nonempty closed and convex subset of a uniformly smooth and uniformly convex real Banach space $E$ with dual space $E^{*}$. We introduce a hybrid projection method for approximating a common element of four classes of constraints: the set of fixed points of a countable family of generalized nonexpansive-type maps, the solution sets of countably many equilibrium problems, the solution sets of countably many variational inequality problems, and the zero sets of countably many maximal monotone operators. The algorithm combines equilibrium resolvents, variational inequality resolvents, generalized resolvents of maximal monotone operators and a shrinking projection step. Under precise monotonicity, continuity and closedness assumptions, we prove that the generated sequence converges strongly to the generalized projection of the initial point onto the common solution set. We also establish residual convergence, derive convex minimization consequences, present a finite-truncation result, and give an illustrative Hilbert-space specialization showing why the countable setting cannot, in general, be reduced to a finite-family theorem.
\end{abstract}

\section{Introduction}

Let $E$ be a real Banach space with topological dual $E^{*}$ and let $C$ be a nonempty closed and convex subset of $E$. The variational inequality problem associated with a mapping $A:C\to E^{*}$ is to find a point $x^{*}\in C$ such that
\begin{equation}\label{introVI}
\langle y-x^{*},Ax^{*}\rangle\geq 0,\qquad \forall y\in C.
\end{equation}
The set of solutions of \eqref{introVI} is denoted by $VI(C,A)$. This problem, initiated in the classical work of Stampacchia \cite{Stampacchia1964}, is important in nonlinear analysis because it contains, as special cases, constrained optimization problems, complementarity problems and several models arising from mechanics and economics.

Let $f:JC\times JC\to \mathbb{R}$ be a bifunction, where $J$ denotes the normalized duality mapping on $E$. The equilibrium problem considered in this paper is to find a point $x^{*}\in C$ such that
\begin{equation}\label{introEP}
f(Jx^{*},Jy)\geq 0,\qquad \forall y\in C.
\end{equation}
The solution set of \eqref{introEP} is denoted by $EP(f)$. Equilibrium problems, introduced in a systematic form by Blum and Oettli \cite{BlumOettli1994} and further developed in iterative and variational frameworks such as \cite{CombettesHirstoaga2005}, provide a unified formulation for optimization problems, variational inequality problems, saddle point problems and Nash equilibrium problems.

Another closely related problem is the monotone inclusion problem. Its foundations in operator theory go back to Minty and Rockafellar; see, for example, \cite{Minty1962,Rockafellar1970}. Let $M:E\to 2^{E^{*}}$ be a maximal monotone operator. The zero problem for $M$ is to find $x\in E$ such that
\begin{equation}\label{introM}
0\in Mx.
\end{equation}
The solution set of \eqref{introM} is denoted by $M^{-1}0$ or $\operatorname{zer}M$. This class of problems is broad. For example, if $M=\partial h$, where $h:E\to (-\infty,+\infty]$ is a proper, convex and lower semicontinuous function, then \eqref{introM} is equivalent to the convex minimization problem
\[
\min_{x\in E}h(x).
\]
Thus, any algorithm which treats zeros of maximal monotone operators together with fixed point, equilibrium and variational inequality constraints has a natural relevance to convex optimization and nonlinear operator theory.

In a smooth Banach space, the Lyapunov functional $\phi:E\times E\to\mathbb{R}$ is defined by
\begin{equation}\label{phi}
\phi(x,y)=\|x\|^{2}-2\langle x,Jy\rangle+\|y\|^{2},\qquad x,y\in E.
\end{equation}
If $E$ is a Hilbert space, then $J$ is the identity map and $\phi(x,y)=\|x-y\|^{2}$. Hence, $\phi$ plays the role of the square of the norm in Banach spaces and is one of the main tools in the study of strong convergence theorems in smooth Banach spaces.

Recently, several authors have studied hybrid and projection-type methods for fixed point problems, variational inequality problems, maximal monotone inclusions and equilibrium problems in Banach spaces; see, for instance, \cite{ChidumeIdu2016,ChidumeOtuboEzeaUba2017,KlineamSuantaiTakahashi2012}. In particular, hybrid methods are useful because they often yield strong convergence. In \cite{UbaCarpathian2023}, a hybrid scheme was introduced for approximating a common element of the set of $J$-fixed points of a countable family of generalized $J_{*}$-nonexpansive maps, the common solution set of a finite family of variational inequality problems and the common solution set of a finite family of equilibrium problems. However, the variational inequality and equilibrium components in that work were finite families, and monotone inclusion problems governed by maximal monotone operators were not included in the algorithm.

The purpose of the present paper is to combine two natural extensions of that result. First, we replace the finite families of equilibrium and variational inequality problems by countable families. Secondly, we incorporate the zero problem for a countable family of maximal monotone operators. Thus, the target set in this paper is of the form
\[
B=F_J(\Gamma)\cap\left(\bigcap_{i=1}^{\infty}EP(f_i)\right)
\cap\left(\bigcap_{j=1}^{\infty}VI(C,A_j)\right)
\cap\left(\bigcap_{k=1}^{\infty}M_k^{-1}0\right),
\]
where $\Gamma$ is a family of generalized $J_{*}$-nonexpansive maps, $\{f_i\}$ is a countable family of equilibrium bifunctions, $\{A_j\}$ is a countable family of continuous monotone mappings and $\{M_k\}$ is a countable family of maximal monotone operators.

The main difficulty in treating countable families is that the algorithm must visit each problem infinitely often. For this reason, we use index maps $\sigma,\tau,\rho:\mathbb{N}\to\mathbb{N}$ such that every positive integer occurs infinitely many times in each sequence. At the $n$th step, the algorithm uses $f_{\sigma(n)}$, $A_{\tau(n)}$ and $M_{\rho(n)}$. This allows the proof to select subsequences for each fixed component and verify membership of the strong limit in every individual solution set by explicit convergence estimates.

The main contribution of this paper is not merely the addition of another hybrid algorithm, nor simply a finite-family theorem rewritten with countably many indices. The countable setting introduces a genuine activation issue: every component problem must be visited infinitely often while the shrinking projection sets remain nonempty, closed, convex and nested. We address this issue by using index maps and by constructing a single shrinking generalized projection scheme which simultaneously treats four countable structures: a countable family of generalized $J_{*}$-nonexpansive mappings, countably many equilibrium problems, countably many variational inequality problems and countably many maximal monotone inclusion problems. The algorithm approximates the generalized projection of the initial point onto the full infinite intersection. We also include a finite-truncation proposition and concrete examples in classical Banach and Hilbert spaces showing that the countable setting cannot, in general, be recovered from a finite-family result.

For clarity, the relationship between the present work and some representative hybrid methods is summarized below. The table is not meant to be exhaustive; it only indicates the particular gap addressed here.
\begin{center}
\footnotesize
\setlength{\tabcolsep}{3pt}
\begin{tabular}{|>{\raggedright\arraybackslash}p{3.1cm}|>{\raggedright\arraybackslash}p{2.8cm}|>{\raggedright\arraybackslash}p{2.3cm}|>{\raggedright\arraybackslash}p{2.3cm}|>{\raggedright\arraybackslash}p{3.0cm}|}
\hline
Reference & Fixed point component & Equilibrium component & Variational inequality component & Maximal monotone component \\ \hline
\cite{UbaCarpathian2023} & countable family & finite family & finite family & absent \\ \hline
\cite{ZegeyeShahzad2011,ZegeyeShahzad2014} & fixed point component present & finite-family framework & finite-family framework & absent from the main scheme considered there \\ \hline
Present paper & countable generalized $J_{*}$-nonexpansive family & countable family & countable family & countable family \\ \hline
\end{tabular}
\end{center}

\section{Preliminaries}

Throughout this paper, $E$ will denote a real Banach space with dual space $E^{*}$. We denote by $x_n\rightharpoonup x$ and $x_n\to x$ the weak and strong convergence of $\{x_n\}$ to $x$, respectively. The normalized duality mapping $J:E\to 2^{E^{*}}$ is defined by
\[
Jx:=\{x^{*}\in E^{*}:\langle x,x^{*}\rangle=\|x\|\|x^{*}\|,\ \|x\|=\|x^{*}\|\}.
\]
It is well known from the geometry of Banach spaces that if $E$ is smooth, strictly convex and reflexive, then $J$ is single-valued, one-to-one and onto; see, for example, \cite{Cioranescu1990,Takahashi2000}. In this case the inverse duality mapping from $E^{*}$ into $E$ will be denoted by
\[
J_{*}=J^{-1}:E^{*}\to E.
\]
Moreover, if $E$ is uniformly smooth, then $J$ is uniformly continuous on bounded subsets of $E$; if $E$ is uniformly smooth and uniformly convex, then $J_{*}=J^{-1}$ is uniformly continuous on bounded subsets of $E^{*}$.

Let $E$ be a smooth real Banach space. The Lyapunov functional $\phi:E\times E\to\mathbb{R}$ is defined by \eqref{phi}, and has the following property
\begin{equation}\label{phibounds}
(\|x\|-\|y\|)^2\leq \phi(x,y)\leq (\|x\|+\|y\|)^2,
\end{equation}
for all $x,y\in E$.

\begin{definition}
Let $T:C\to E^{*}$ be a mapping. A point $p\in C$ is called a $J$-fixed point of $T$ if
\[
Tp=Jp.
\]
The set of $J$-fixed points of $T$ will be denoted by $F_J(T)$.
If $\Gamma$ is a family of mappings from $C$ into $E^{*}$, we write
\[
F_J(\Gamma):=\bigcap_{T\in \Gamma}F_J(T).
\]
\end{definition}

\begin{definition}
A map $T:C\to E^{*}$ is called generalized $J_{*}$-nonexpansive if $F_J(T)\neq\emptyset$ and
\[
\phi(p,(J_{*}\circ T)x)\leq \phi(p,x),\qquad \forall x\in C,\ p\in F_J(T).
\]
\end{definition}

\begin{definition}
A map $T:C\to E^{*}$ is called $J_{*}$-closed if $J_{*}\circ T:C\to E$ is closed, that is, whenever $x_n\to x$ and $(J_{*}\circ T)x_n\to y$, then
\[
y=(J_{*}\circ T)x.
\]
\end{definition}

For the equilibrium problems, we shall assume that each bifunction $f_i:JC\times JC\to\mathbb{R}$ satisfies the following conditions:
\begin{itemize}
\item[(A1)] $f_i(x^{*},x^{*})=0$ for all $x^{*}\in JC$;
\item[(A2)] $f_i$ is monotone, that is,
\[
f_i(x^{*},y^{*})+f_i(y^{*},x^{*})\leq 0,
\]
for all $x^{*},y^{*}\in JC$;
\item[(A3)] for all $x^{*},y^{*},z^{*}\in JC$,
\[
\limsup_{t\downarrow 0}f_i(tz^{*}+(1-t)x^{*},y^{*})\leq f_i(x^{*},y^{*});
\]
\item[(A4)] for each $x^{*}\in JC$, $f_i(x^{*},\cdot)$ is convex and lower semicontinuous.
\end{itemize}

\begin{definition}
Let $M:E\to 2^{E^{*}}$ be a multivalued map. The graph of $M$ is defined by
\[
G(M)=\{(x,x^{*})\in E\times E^{*}:x^{*}\in Mx\}.
\]
The operator $M$ is called monotone if
\[
\langle x-y,x^{*}-y^{*}\rangle\geq 0,
\]
for all $(x,x^{*}),(y,y^{*})\in G(M)$. It is called maximal monotone if its graph is not properly contained in the graph of any other monotone operator.
\end{definition}

Let $M:E\to 2^{E^{*}}$ be a maximal monotone operator. For $r>0$, the generalized resolvent of $M$ is denoted by
\[
G_r^M=(J+rM)^{-1}J.
\]
Thus, if $q=G_r^Mx$, then
\begin{equation}\label{resolventidentity}
Jx\in Jq+rMq,
\end{equation}
or equivalently,
\[
\frac{Jx-Jq}{r}\in Mq.
\]

We shall use the following standard facts. 

\begin{lemma}[Uniform convexity inequality \cite{KamimuraTakahashi2002}]\label{glemma}
Let $E$ be a uniformly convex Banach space and let $R>0$. Then there exists a continuous, strictly increasing and convex function $g:[0,2R)\to[0,\infty)$ with $g(0)=0$ such that, for every integer $N\geq2$, every finite family $x_1,x_2,\ldots,x_N\in B_R(0)$ and every choice of positive numbers $\lambda_1,\lambda_2,\ldots,\lambda_N$ with $\sum_{i=1}^{N}\lambda_i=1$, one has, for all $i,j\in\{1,2,\ldots,N\}$ with $i<j$,
\[
\left\|\sum_{n=1}^{N}\lambda_nx_n\right\|^2\leq \sum_{n=1}^{N}\lambda_n\|x_n\|^2-\lambda_i\lambda_jg(\|x_i-x_j\|).
\]
\end{lemma}

\begin{lemma}[Alber-type convergence criterion \cite{Alber1996,KamimuraTakahashi2002}]\label{phitozero}
Let $E$ be a real smooth and uniformly convex Banach space, and let $\{x_n\}$ and $\{y_n\}$ be two sequences in $E$. If either $\{x_n\}$ or $\{y_n\}$ is bounded and $\phi(x_n,y_n)\to 0$ as $n\to\infty$, then
\[
\|x_n-y_n\|\to 0.
\]
\end{lemma}

\begin{definition}
Let $C$ be a nonempty closed and convex subset of a smooth, strictly convex and reflexive Banach space $E$. The generalized projection of $x\in E$ onto $C$ is the unique point
\[
\Pi_Cx=\operatorname*{argmin}_{y\in C}\phi(y,x).
\]
\end{definition}

The following characterization of the generalized projection will be used repeatedly; see, for example, Alber \cite{Alber1996} and Kamimura--Takahashi \cite{KamimuraTakahashi2002}.

\begin{lemma}\label{projectionlemma}
Let $C$ be a nonempty closed and convex subset of a smooth, strictly convex and reflexive Banach space $E$. Then $z=\Pi_Cx$ if and only if
\[
\langle y-z,Jx-Jz\rangle\leq 0,\qquad \forall y\in C.
\]
Moreover,
\[
\phi(y,z)+\phi(z,x)\leq \phi(y,x),\qquad \forall y\in C.
\]
\end{lemma}

\begin{lemma}[Equilibrium resolvent \cite{TakahashiZembayashi2008}]\label{EPrezolvent}
Let $C$ be a nonempty closed subset of a smooth, strictly convex and reflexive Banach space $E$ such that $JC$ is closed and convex. Let $f:JC\times JC\to\mathbb{R}$ satisfy $(A1)-(A4)$. For $r>0$ and $x\in E$, define $T_r^f:E\to C$ by
\[
T_r^f x=\left\{z\in C:f(Jz,Jy)+\frac{1}{r}\langle y-z,Jz-Jx\rangle\geq 0,\ \forall y\in C\right\}.
\]
Then $T_r^f$ is single valued, $F(T_r^f)=EP(f)$, $JEP(f)$ is closed and convex, and
\[
\phi(p,T_r^f x)+\phi(T_r^f x,x)\leq \phi(p,x),\qquad \forall p\in EP(f).
\]
\end{lemma}

\begin{lemma}[Variational-inequality resolvent \cite{ZegeyeShahzad2014}]\label{VIrezolvent}
Let $C$ be a nonempty closed subset of a smooth, strictly convex and reflexive Banach space $E$ such that $JC$ is closed and convex. Let $A:C\to E^{*}$ be continuous and monotone. For $r>0$ and $x\in E$, define the mapping $F_r^A:E\to C$ by
\[
F_r^A x=\left\{z\in C:\langle y-z,Az\rangle+\frac{1}{r}\langle y-z,Jz-Jx\rangle\geq 0,\ \forall y\in C\right\}.
\]
Then $F_r^A$ is single valued, $F(F_r^A)=VI(C,A)$, $JVI(C,A)$ is closed and convex, and
\[
\phi(p,F_r^A x)+\phi(F_r^A x,x)\leq \phi(p,x),\qquad \forall p\in VI(C,A).
\]
\end{lemma}

\begin{lemma}[Generalized resolvent \cite{KamimuraKohsakaTakahashi2004,KohsakaTakahashi2004,AoyamaKimuraTakahashi2008}]\label{MMresolvent}
Let $E$ be a uniformly smooth and uniformly convex real Banach space, and let
$J:E\to E^*$ denote the normalized duality mapping. Let
$M:E\to 2^{E^{*}}$ be maximal monotone with $M^{-1}0\neq\emptyset$. For
$r>0$, define
\[
G_r^M=(J+rM)^{-1}J.
\]
Then $G_r^M:E\to D(M)$ is single valued,
\[
F(G_r^M)=M^{-1}0,
\]
and
\[
\phi(p,G_r^Mx)+\phi(G_r^Mx,x)\leq \phi(p,x),
\qquad x\in E,\quad p\in M^{-1}0.
\]
\end{lemma}

\begin{lemma}[Closedness of maximal monotone graphs \cite{Rockafellar1970,AlberRyazantseva2006}]\label{closedgraphM}
Let $M:E\to 2^{E^{*}}$ be maximal monotone. If
\[
x_n\to x,\qquad x_n^{*}\to x^{*},\qquad x_n^{*}\in Mx_n\quad(n\geq1),
\]
then $x^{*}\in Mx$.
\end{lemma}

\begin{remark}
In the sequel, the resolvents $T_r^{f_i}$, $F_r^{A_j}$ and $G_r^{M_k}$ will be used repeatedly. To simplify notation, we shall write
\[
T_{r_n}^{\sigma(n)}:=T_{r_n}^{f_{\sigma(n)}},\qquad
F_{r_n}^{\tau(n)}:=F_{r_n}^{A_{\tau(n)}},\qquad
G_{r_n}^{\rho(n)}:=G_{r_n}^{M_{\rho(n)}}.
\]
\end{remark}

\section{Main results}

Let $E$ be a uniformly smooth and uniformly convex real Banach space with dual space $E^{*}$, and let $C$ be a nonempty closed and convex subset of $E$ such that $JC$ is closed and convex. Let $\{f_i\}_{i=1}^{\infty}$ be a countable family of bifunctions from $JC\times JC$ into $\mathbb{R}$ satisfying $(A1)-(A4)$. Let $\{A_j\}_{j=1}^{\infty}$ be a countable family of continuous monotone mappings from $C$ into $E^{*}$. Let $\{M_k\}_{k=1}^{\infty}$ be a countable family of maximal monotone operators from $E$ into $2^{E^{*}}$ such that $G_r^{M_k}(C)\subset C$ for each $r>0$ and each $k\in\mathbb{N}$. Let $T_n:C\to E^{*}$, $n\geq1$, be a countable family of generalized $J_{*}$-nonexpansive maps and let $\Gamma$ be a family of $J_{*}$-closed and generalized $J_{*}$-nonexpansive maps from $C$ into $E^{*}$ such that
\[
\bigcap_{n=1}^{\infty}F_J(T_n)=F_J(\Gamma)\neq\emptyset.
\]

Let $\sigma,\tau,\rho:\mathbb{N}\to\mathbb{N}$ be index maps such that for every $m\in\mathbb{N}$, the sets
\[
\{n\in\mathbb{N}:\sigma(n)=m\},\quad
\{n\in\mathbb{N}:\tau(n)=m\},\quad
\{n\in\mathbb{N}:\rho(n)=m\}
\]
are infinite. A typical choice is the repeated triangular ordering
\[
1,1,2,1,2,3,1,2,3,4,\ldots.
\]
Equivalently, one may define an index map $\eta:\mathbb N\to\mathbb N$ by
\[
\eta\left(\frac{m(m-1)}{2}+r\right)=r,
\qquad m\in\mathbb N,\quad r=1,2,\ldots,m.
\]
Then each positive integer occurs infinitely many times in $\{\eta(n)\}$. In applications one may take $\sigma=\tau=\rho=\eta$, or use three maps of this type.

Define
\begin{equation}\label{Bset}
B:=F_J(\Gamma)\cap\left(\bigcap_{i=1}^{\infty}EP(f_i)\right)
\cap\left(\bigcap_{j=1}^{\infty}VI(C,A_j)\right)
\cap\left(\bigcap_{k=1}^{\infty}M_k^{-1}0\right).
\end{equation}

Let the sequence $\{x_n\}$ be generated by the following algorithm:
\begin{equation}\label{alg}
\left\{
\begin{array}{ll}
x_1=x\in C,\quad C_1=C,\\[0.1cm]
z_n=T_{r_n}^{\sigma(n)}x_n,\\[0.1cm]
u_n=F_{r_n}^{\tau(n)}x_n,\\[0.1cm]
q_n=G_{r_n}^{\rho(n)}x_n,\\[0.1cm]
y_n=J^{-1}\left(\alpha_1Jx_n+\alpha_2Jz_n+\alpha_3Ju_n+\alpha_4Jq_n+\alpha_5T_nq_n\right),\\[0.1cm]
C_{n+1}=\{v\in C_n:\phi(v,y_n)\leq \phi(v,x_n)\},\\[0.1cm]
x_{n+1}=\Pi_{C_{n+1}}x,
\end{array}
\right.
\end{equation}
for all $n\in\mathbb{N}$, where $\alpha_i\in(0,1)$, $i=1,2,3,4,5$, satisfy $\sum_{i=1}^{5}\alpha_i=1$, and $\{r_n\}\subset[a,\infty)$ for some $a>0$.

\begin{lemma}\label{welldefined}
Assume that $B\neq\emptyset$. Then the sets $C_n$ generated by \eqref{alg} are nonempty, closed and convex, and the sequence $\{x_n\}$ is well defined.
\end{lemma}

\begin{proof}
We prove by induction that $B\subset C_n$ and that $C_n$ is closed and convex for every $n\geq1$. The assertion is true for $n=1$ because $C_1=C$. Suppose it is true for some $n\geq1$ and fix $p\in B$. Then
\[
p\in EP(f_{\sigma(n)}),\qquad p\in VI(C,A_{\tau(n)}),\qquad
p\in M_{\rho(n)}^{-1}0,
\]
and, since $\bigcap_{m=1}^{\infty}F_J(T_m)=F_J(\Gamma)$, also $p\in F_J(T_n)$. Lemmas \ref{EPrezolvent}, \ref{VIrezolvent} and \ref{MMresolvent} give
\[
\phi(p,z_n)\leq\phi(p,x_n),\qquad
\phi(p,u_n)\leq\phi(p,x_n),\qquad
\phi(p,q_n)\leq\phi(p,x_n).
\]
Since $T_n$ is generalized $J_{*}$-nonexpansive and $p\in F_J(T_n)$,
\[
\phi(p,J^{-1}T_nq_n)\leq \phi(p,q_n)\leq \phi(p,x_n).
\]
Using the convexity of the norm squared in $E^{*}$, Lemma \ref{glemma}, and the identity
\[
Jy_n=\alpha_1Jx_n+\alpha_2Jz_n+\alpha_3Ju_n+\alpha_4Jq_n+\alpha_5T_nq_n,
\]
we obtain
\[
\phi(p,y_n)\leq
\alpha_1\phi(p,x_n)+\alpha_2\phi(p,z_n)+\alpha_3\phi(p,u_n)
+\alpha_4\phi(p,q_n)+\alpha_5\phi(p,J^{-1}T_nq_n)
\leq \phi(p,x_n).
\]
Thus $p\in C_{n+1}$, and so $B\subset C_{n+1}$. In particular, $C_{n+1}$ is nonempty.

It remains to check convexity and closedness. From the definition of $C_{n+1}$ and \eqref{phi}, the inequality $\phi(v,y_n)\leq\phi(v,x_n)$ is equivalent to
\[
2\langle v,Jx_n-Jy_n\rangle\leq \|x_n\|^2-\|y_n\|^2.
\]
Hence
\[
C_{n+1}=C_n\cap
\left\{v\in E:2\langle v,Jx_n-Jy_n\rangle\leq\|x_n\|^2-\|y_n\|^2\right\},
\]
which is the intersection of $C_n$ and a closed half-space. Therefore $C_{n+1}$ is closed and convex. By induction, all $C_n$ are nonempty, closed and convex. Since $E$ is smooth, strictly convex and reflexive, the generalized projection $\Pi_{C_n}x$ exists and is unique for every $n$. Thus \eqref{alg} is well defined.
\end{proof}

\begin{theorem}\label{mainthm}
Let $E$ be a uniformly smooth and uniformly convex real Banach space with dual space $E^{*}$, and let $C$ be a nonempty closed and convex subset of $E$ such that $JC$ is closed and convex. Let $\{f_i\}_{i=1}^{\infty}$ be bifunctions from $JC\times JC$ into $\mathbb{R}$ satisfying $(A1)$--$(A4)$, let $\{A_j\}_{j=1}^{\infty}$ be continuous monotone mappings from $C$ into $E^{*}$, and let $\{M_k\}_{k=1}^{\infty}$ be maximal monotone operators from $E$ into $2^{E^{*}}$ such that $G_r^{M_k}(C)\subset C$ for all $r>0$ and $k\in\mathbb{N}$. Let $\{T_n\}_{n=1}^{\infty}$ be generalized $J_{*}$-nonexpansive mappings from $C$ into $E^{*}$, and let $\Gamma$ be a family of $J_{*}$-closed generalized $J_{*}$-nonexpansive mappings from $C$ into $E^{*}$ such that
\[
\bigcap_{n=1}^{\infty}F_J(T_n)=F_J(\Gamma)\neq\emptyset.
\]
Let $\sigma,\tau,\rho:\mathbb N\to\mathbb N$ be index maps for which every positive integer occurs infinitely many times in each of the sequences $\{\sigma(n)\}$, $\{\tau(n)\}$ and $\{\rho(n)\}$. Assume $B$ defined by \eqref{Bset} is nonempty, closed and convex. Let $\{r_n\}\subset [a,\infty)$ for some $a>0$, and let $\alpha_i\in(0,1)$, $i=1,2,3,4,5$, satisfy $\sum_{i=1}^{5}\alpha_i=1$. Suppose that $\{x_n\}$ is generated by \eqref{alg} and that $\{T_n\}$ satisfies the NST-condition with $\Gamma$, namely, for every bounded sequence $\{s_n\}\subset C$,
\[
\|Js_n-T_ns_n\|\to0\quad \Longrightarrow\quad \|Js_n-Ts_n\|\to0,
\qquad \forall T\in\Gamma.
\]
Then $\{x_n\}$ converges strongly to $\Pi_Bx$, the generalized projection of the initial point $x$ onto $B$.
\end{theorem}

\begin{proof}
By Lemma \ref{welldefined}, the sequence is well defined and $B\subset C_n$ for every $n\geq1$. The proof is divided into five steps.

\noindent {\bf Step 1: Boundedness and Cauchy property.}
Fix $p\in B$. Since $x_n=\Pi_{C_n}x$ and $p\in B\subset C_n$, Lemma \ref{projectionlemma} gives
\[
\phi(x_n,x)\leq \phi(p,x),\qquad n\geq1.
\]
Therefore $\{x_n\}$ is bounded. Also, because $x_{n+1}\in C_{n+1}\subset C_n$ and $x_n=\Pi_{C_n}x$, we have
\[
\phi(x_n,x)\leq \phi(x_{n+1},x),\qquad n\geq1.
\]
Thus $\{\phi(x_n,x)\}$ is increasing and bounded, so it has a finite limit.

Let $m>n$. Since $x_m\in C_m\subset C_n$, Lemma \ref{projectionlemma} yields
\begin{equation}\label{cauchyphi-revised}
\phi(x_m,x_n)+\phi(x_n,x)\leq \phi(x_m,x).
\end{equation}
Letting $m,n\to\infty$ in \eqref{cauchyphi-revised}, we obtain $\phi(x_m,x_n)\to0$. By Lemma \ref{phitozero}, $\|x_m-x_n\|\to0$. Hence $\{x_n\}$ is Cauchy. Since $C$ is closed, there exists $x^{*}\in C$ such that
\begin{equation}\label{xconv-revised}
x_n\to x^{*}.
\end{equation}

\noindent {\bf Step 2: Convergence of the auxiliary sequences.}
Since $x_{n+1}\in C_{n+1}$, the definition of $C_{n+1}$ gives
\[
\phi(x_{n+1},y_n)\leq \phi(x_{n+1},x_n).
\]
By \eqref{cauchyphi-revised}, with $m=n+1$, $\phi(x_{n+1},x_n)\to0$. Hence $\phi(x_{n+1},y_n)\to0$, and Lemma \ref{phitozero} implies
\begin{equation}\label{ynconv-revised}
y_n-x_{n+1}\to0.
\end{equation}
Together with \eqref{xconv-revised}, this gives
\begin{equation}\label{ynx-revised}
y_n\to x^{*}.
\end{equation}

The sequences $\{z_n\}$, $\{u_n\}$, $\{q_n\}$ and $\{J^{-1}T_nq_n\}$ are bounded. Indeed, the resolvent inequalities and the generalized $J_{*}$-nonexpansiveness of $T_n$ give, for $p\in B$,
\[
\phi(p,z_n),\ \phi(p,u_n),\ \phi(p,q_n),\ \phi(p,J^{-1}T_nq_n)\leq \phi(p,x_n),
\]
and the right-hand side is bounded.

Since $E$ is uniformly smooth, $E^{*}$ is uniformly convex. Choose $R>0$ such that
\[
Jx_n,\quad Jz_n,\quad Ju_n,\quad Jq_n,\quad T_nq_n\in B_R(0)\subset E^{*},
\qquad n\geq1.
\]
Set
\[
w_{1,n}=Jx_n,\quad w_{2,n}=Jz_n,\quad w_{3,n}=Ju_n,\quad
w_{4,n}=Jq_n,\quad w_{5,n}=T_nq_n.
\]
Applying Lemma \ref{glemma} in the space $E^{*}$ with this fixed radius $R$, we obtain a continuous, strictly increasing and convex function $g$ with $g(0)=0$. Using the estimate from Lemma \ref{glemma} with the pair $(w_{1,n},w_{\ell,n})$, $\ell=2,3,4,5$, and then using the resolvent inequalities together with the generalized $J_{*}$-nonexpansiveness of $T_n$, we obtain, for each $\ell=2,3,4,5$,
\begin{equation}\label{gterms-revised}
\alpha_1\alpha_{\ell} g(\|w_{1,n}-w_{\ell,n}\|)
\leq \phi(p,x_n)-\phi(p,y_n).
\end{equation}
Inequality \eqref{gterms-revised} follows from
\[
\phi(p,y_n)
\leq \sum_{i=1}^{5}\alpha_i\phi(p,J^{-1}w_{i,n})
      -\alpha_1\alpha_{\ell}g(\|w_{1,n}-w_{\ell,n}\|)
\leq \phi(p,x_n)-\alpha_1\alpha_{\ell}g(\|w_{1,n}-w_{\ell,n}\|),
\]
where $J^{-1}w_{5,n}=J^{-1}T_nq_n$.
By \eqref{xconv-revised} and \eqref{ynx-revised}, the right-hand side of \eqref{gterms-revised} tends to zero. Since $g$ is strictly increasing and $g(0)=0$, it follows that
\[
\|Jx_n-Jz_n\|\to0,\quad \|Jx_n-Ju_n\|\to0,\quad
\|Jx_n-Jq_n\|\to0,\quad \|Jx_n-T_nq_n\|\to0.
\]
The inverse duality map $J^{-1}=J_{*}$ is uniformly continuous on bounded subsets of $E^{*}$. Therefore,
\begin{equation}\label{zuqconv-revised}
z_n\to x^{*},\qquad u_n\to x^{*},\qquad q_n\to x^{*}.
\end{equation}
Moreover,
\begin{equation}\label{Treg-revised}
\|Jq_n-T_nq_n\|\leq \|Jq_n-Jx_n\|+\|Jx_n-T_nq_n\|\to0.
\end{equation}

\noindent {\bf Step 3: The limit belongs to the common $J$-fixed point set.}
From \eqref{zuqconv-revised} and \eqref{Treg-revised}, the bounded sequence $\{q_n\}$ satisfies $\|Jq_n-T_nq_n\|\to0$. By the NST-condition with $\Gamma$,
\[
\|Jq_n-Tq_n\|\to0,\qquad \forall T\in\Gamma.
\]
Fix $T\in\Gamma$. Since $q_n\to x^{*}$ and $Jq_n\to Jx^{*}$, we get $Tq_n\to Jx^{*}$ and hence $(J_{*}\circ T)q_n\to x^{*}$. The $J_{*}$-closedness of $T$ gives $(J_{*}\circ T)x^{*}=x^{*}$, equivalently $Tx^{*}=Jx^{*}$. Thus
\begin{equation}\label{xFJ-revised}
x^{*}\in F_J(\Gamma).
\end{equation}

\noindent {\bf Step 4: The limit solves every equilibrium, variational inequality and inclusion problem.}
We give the details of the three membership arguments by writing out the estimates used to identify the strong limit.

First fix $i\in\mathbb N$. Since every integer occurs infinitely often in $\{\sigma(n)\}$, there is a subsequence $\{n_l\}$ such that $\sigma(n_l)=i$ for all $l$. From
$z_{n_l}=T_{r_{n_l}}^{f_i}x_{n_l}$, we have
\begin{equation}\label{EPpass-revised}
f_i(Jz_{n_l},Jy)+\frac{1}{r_{n_l}}\langle y-z_{n_l},Jz_{n_l}-Jx_{n_l}\rangle\geq0,
\qquad \forall y\in C.
\end{equation}
Fix $y\in C$. By the monotonicity of $f_i$,
\[
f_i(Jy,Jz_{n_l})\leq -f_i(Jz_{n_l},Jy).
\]
Combining this with \eqref{EPpass-revised} gives
\begin{equation}\label{EPestimate-detailed}
f_i(Jy,Jz_{n_l})\leq
\frac{1}{r_{n_l}}\langle y-z_{n_l},Jz_{n_l}-Jx_{n_l}\rangle.
\end{equation}
Because $r_{n_l}\geq a>0$, the sequence $\{z_{n_l}\}$ is bounded and
$\|Jz_{n_l}-Jx_{n_l}\|\to0$, we have
\[
\left|\frac{1}{r_{n_l}}\langle y-z_{n_l},Jz_{n_l}-Jx_{n_l}\rangle\right|
\leq \frac{1}{a}\|y-z_{n_l}\|\,\|Jz_{n_l}-Jx_{n_l}\|\to0.
\]
Hence \eqref{EPestimate-detailed} implies
\begin{equation}\label{EPlimsup-detail}
\limsup_{l\to\infty} f_i(Jy,Jz_{n_l})\leq0.
\end{equation}
On the other hand, \eqref{zuqconv-revised} gives $z_{n_l}\to x^{*}$; since $J$ is norm-to-norm continuous on bounded sets, $Jz_{n_l}\to Jx^{*}$. The lower semicontinuity of $f_i(Jy,\cdot)$ now yields
\[
f_i(Jy,Jx^{*})\leq \liminf_{l\to\infty} f_i(Jy,Jz_{n_l}).
\]
Together with \eqref{EPlimsup-detail}, this gives
\begin{equation}\label{EPdualineq-detail}
f_i(Jy,Jx^{*})\leq0,
\qquad \forall y\in C.
\end{equation}
For $t\in(0,1]$, set $w_t^{*}=tJy+(1-t)Jx^{*}\in JC$ and choose $y_t\in C$ with $Jy_t=w_t^{*}$. Applying \eqref{EPdualineq-detail} with $y=y_t$ gives $f_i(Jy_t,Jx^{*})\leq0$. Since $f_i(Jy_t,Jy_t)=0$ and $f_i(Jy_t,\cdot)$ is convex,
\[
0=f_i(Jy_t,Jy_t)\leq t f_i(Jy_t,Jy)+(1-t)f_i(Jy_t,Jx^{*})\leq t f_i(Jy_t,Jy).
\]
Thus $f_i(Jy_t,Jy)\geq0$. Assumption $(A3)$ gives
\[
f_i(Jx^{*},Jy)
\geq \limsup_{t\downarrow0} f_i(tJy+(1-t)Jx^{*},Jy)
=\limsup_{t\downarrow0} f_i(Jy_t,Jy)\geq0.
\]
Therefore $x^{*}\in EP(f_i)$. Since $i$ was arbitrary,
\begin{equation}\label{xEP-revised}
x^{*}\in\bigcap_{i=1}^{\infty}EP(f_i).
\end{equation}

Next fix $j\in\mathbb N$ and take a subsequence $\{m_l\}$ with $\tau(m_l)=j$ for all $l$. Since $u_{m_l}=F_{r_{m_l}}^{A_j}x_{m_l}$,
\begin{equation}\label{VIestimate-start-detail}
\langle y-u_{m_l},A_ju_{m_l}\rangle+
\frac{1}{r_{m_l}}\langle y-u_{m_l},Ju_{m_l}-Jx_{m_l}\rangle\geq0,
\qquad \forall y\in C.
\end{equation}
Fix $y\in C$ and put
\[
\varepsilon_l(y)=\frac{1}{r_{m_l}}\langle y-u_{m_l},Ju_{m_l}-Jx_{m_l}\rangle.
\]
Since $r_{m_l}\geq a>0$, $\{u_{m_l}\}$ is bounded and $\|Ju_{m_l}-Jx_{m_l}\|\to0$, we have
\[
|\varepsilon_l(y)|\leq \frac{1}{a}\|y-u_{m_l}\|\,\|Ju_{m_l}-Jx_{m_l}\|\to0.
\]
Also, $u_{m_l}\to x^{*}$ and the continuity of $A_j$ imply $A_ju_{m_l}\to A_jx^{*}$ in $E^{*}$. Therefore
\[
\langle y-u_{m_l},A_ju_{m_l}\rangle\to \langle y-x^{*},A_jx^{*}\rangle.
\]
If $\langle y-x^{*},A_jx^{*}\rangle<0$, then for all sufficiently large $l$ the sum in \eqref{VIestimate-start-detail} would be negative, contradicting \eqref{VIestimate-start-detail}. Hence
\[
\langle y-x^{*},A_jx^{*}\rangle\geq0,
\qquad \forall y\in C.
\]
Thus $x^{*}\in VI(C,A_j)$. Since $j$ was arbitrary,
\begin{equation}\label{xVI-revised}
x^{*}\in\bigcap_{j=1}^{\infty}VI(C,A_j).
\end{equation}

Finally fix $k\in\mathbb N$ and choose a subsequence $\{s_l\}$ such that $\rho(s_l)=k$ for all $l$. Since $q_{s_l}=G_{r_{s_l}}^{M_k}x_{s_l}$, the resolvent identity gives
\[
\eta_l:=\frac{Jx_{s_l}-Jq_{s_l}}{r_{s_l}}\in M_kq_{s_l}.
\]
The sequence $\{q_{s_l}\}$ converges to $x^{*}$ by \eqref{zuqconv-revised}. Moreover, using $r_{s_l}\geq a>0$ and $\|Jx_{s_l}-Jq_{s_l}\|\to0$, we obtain
\[
\|\eta_l\|\leq \frac{1}{a}\|Jx_{s_l}-Jq_{s_l}\|\to0.
\]
Thus $(q_{s_l},\eta_l)\in G(M_k)$, $q_{s_l}\to x^{*}$ and $\eta_l\to0$. By the closedness of the graph of $M_k$ stated in Lemma \ref{closedgraphM}, $0\in M_kx^{*}$. Hence $x^{*}\in M_k^{-1}0$. Since $k$ was arbitrary,
\begin{equation}\label{xM-revised}
x^{*}\in\bigcap_{k=1}^{\infty}M_k^{-1}0.
\end{equation}
Combining \eqref{xFJ-revised}, \eqref{xEP-revised}, \eqref{xVI-revised} and \eqref{xM-revised}, we obtain $x^{*}\in B$.

\noindent {\bf Step 5: Identification of the limit.}
Let $b=\Pi_Bx$. Since $b\in B\subset C_n$ and $x_n=\Pi_{C_n}x$, we have
\[
\phi(x_n,x)\leq \phi(b,x),\qquad n\geq1.
\]
Letting $n\to\infty$ gives $\phi(x^{*},x)\leq\phi(b,x)$. On the other hand, $x^{*}\in B$ and $b=\Pi_Bx$, so
\[
\phi(b,x)\leq \phi(x^{*},x).
\]
Consequently, $\phi(x^{*},x)=\phi(b,x)$. The generalized projection onto a nonempty closed and convex set is unique; hence $x^{*}=b=\Pi_Bx$. This completes the proof.
\end{proof}

\begin{proposition}[Asymptotic residuals]\label{residualprop}
Under the assumptions of Theorem \ref{mainthm}, the auxiliary sequences generated by \eqref{alg} satisfy
\[
\|x_n-z_n\|\to0,\qquad \|x_n-u_n\|\to0,\qquad \|x_n-q_n\|\to0,
\]
and
\[
\|Jq_n-T_nq_n\|\to0.
\]
Moreover, for each fixed $i,j,k\in\mathbb N$, along subsequences for which $\sigma(n_l)=i$, $\tau(m_l)=j$ and $\rho(s_l)=k$, respectively, one has
\[
\frac{\|Jz_{n_l}-Jx_{n_l}\|}{r_{n_l}}\to0,\qquad
\frac{\|Ju_{m_l}-Jx_{m_l}\|}{r_{m_l}}\to0,\qquad
\frac{\|Jx_{s_l}-Jq_{s_l}\|}{r_{s_l}}\to0.
\]
These limits express the asymptotic satisfaction of the activated equilibrium, variational inequality and maximal monotone inclusion steps.
\end{proposition}

\begin{proof}
The first three limits follow from Step 2 of the proof of Theorem \ref{mainthm}, where it was shown that
\[
\|Jx_n-Jz_n\|\to0,\qquad \|Jx_n-Ju_n\|\to0,\qquad \|Jx_n-Jq_n\|\to0.
\]
Since $J^{-1}$ is uniformly continuous on bounded subsets of $E^{*}$, these imply $\|x_n-z_n\|\to0$, $\|x_n-u_n\|\to0$ and $\|x_n-q_n\|\to0$. The estimate \eqref{Treg-revised} gives $\|Jq_n-T_nq_n\|\to0$. Finally, because $r_n\geq a>0$, the displayed subsequential limits follow immediately from the preceding convergence of the corresponding duality-map residuals.
\end{proof}

\begin{remark}
Theorem \ref{mainthm} extends the hybrid method of \cite{UbaCarpathian2023} in two directions. First, the finite families of equilibrium and variational inequality problems are replaced by countable families. Secondly, a countable family of maximal monotone inclusion problems is incorporated through generalized resolvents. If the family $\{M_k\}$ is removed and if the index maps $\sigma$ and $\tau$ take values in finite sets, the method reduces to a form of the hybrid scheme studied in \cite{UbaCarpathian2023}.
\end{remark}

\section{Applications and corollaries}

\begin{corollary}\label{withoutMM}
Let the assumptions of Theorem \ref{mainthm} hold, but suppose that the maximal monotone component is absent. Define
\[
B_1=F_J(\Gamma)\cap\left(\bigcap_{i=1}^{\infty}EP(f_i)\right)
\cap\left(\bigcap_{j=1}^{\infty}VI(C,A_j)\right).
\]
Assume that $B_1$ is nonempty, closed and convex.
Let $\{x_n\}$ be generated by
\[
\left\{
\begin{array}{ll}
x_1=x\in C,\quad C_1=C,\\
z_n=T_{r_n}^{\sigma(n)}x_n,\\
u_n=F_{r_n}^{\tau(n)}x_n,\\
y_n=J^{-1}(\alpha_1Jx_n+\alpha_2Jz_n+\alpha_3Ju_n+\alpha_4T_nu_n),\\
C_{n+1}=\{v\in C_n:\phi(v,y_n)\leq\phi(v,x_n)\},\\
x_{n+1}=\Pi_{C_{n+1}}x.
\end{array}
\right.
\]
Assume that $\alpha_i\in(0,1)$ for $i=1,2,3,4$, $\sum_{i=1}^{4}\alpha_i=1$, and $\{T_n\}$ satisfies the NST-condition with $\Gamma$. Then $x_n\to \Pi_{B_1}x$.
\end{corollary}

\begin{proof}
The proof is the same as the proof of Theorem \ref{mainthm} with the maximal monotone resolvent step removed. Indeed, the inequalities obtained from the equilibrium and variational-inequality resolvents and from the generalized $J_{*}$-nonexpansiveness of $T_n$ imply $B_1\subset C_n$ for all $n$. The shrinking projection argument gives that $\{x_n\}$ is Cauchy. The index maps $\sigma$ and $\tau$ allow the same subsequence-selection argument used in Step 4 of Theorem \ref{mainthm}; hence the strong limit belongs to every $EP(f_i)$ and every $VI(C,A_j)$. The NST-condition then gives membership in $F_J(\Gamma)$. Hence the limit belongs to $B_1$, and the uniqueness of the generalized projection identifies it with $\Pi_{B_1}x$.
\end{proof}

\begin{corollary}\label{maxonly}
Let $E$ be a uniformly smooth and uniformly convex real Banach space and let $\{M_k\}_{k=1}^{\infty}$ be a countable family of maximal monotone operators such that
\[
B_2:=\bigcap_{k=1}^{\infty}M_k^{-1}0
\]
is nonempty, closed and convex. Let $\rho:\mathbb{N}\to\mathbb{N}$ be such that every positive integer occurs infinitely often. Let
\[
\left\{
\begin{array}{ll}
x_1=x\in E,\quad C_1=E,\\
q_n=G_{r_n}^{M_{\rho(n)}}x_n,\\
C_{n+1}=\{v\in C_n:\phi(v,q_n)\leq\phi(v,x_n)\},\\
x_{n+1}=\Pi_{C_{n+1}}x,
\end{array}
\right.
\]
where $\{r_n\}\subset[a,\infty)$ for some $a>0$. Then $x_n\to \Pi_{B_2}x$.
\end{corollary}

\begin{proof}
For $p\in B_2$, Lemma \ref{MMresolvent} gives $\phi(p,q_n)\leq\phi(p,x_n)$; hence $B_2\subset C_n$ for every $n$. The same generalized projection argument used in Step 1 of Theorem \ref{mainthm} shows that $\{x_n\}$ is Cauchy and therefore $x_n\to x^{*}$ for some $x^{*}\in E$. Since $x_{n+1}\in C_{n+1}$, one also obtains $q_n-x_n\to0$. Fix $k\in\mathbb N$ and choose a subsequence $\{n_l\}$ such that $\rho(n_l)=k$. From the resolvent identity,
\[
\frac{Jx_{n_l}-Jq_{n_l}}{r_{n_l}}\in M_kq_{n_l}.
\]
Because $r_{n_l}\geq a>0$ and $q_{n_l}\to x^{*}$, the left-hand side converges to $0$. Lemma \ref{closedgraphM} yields $0\in M_kx^{*}$. Since $k$ was arbitrary, $x^{*}\in B_2$. Finally, the minimality property of $x_n=\Pi_{C_n}x$ and the uniqueness of the generalized projection give $x^{*}=\Pi_{B_2}x$.
\end{proof}

\begin{corollary}[Countable convex minimization]\label{convexmin}
Let $E$ be a uniformly smooth and uniformly convex real Banach space, and let $\{h_k\}_{k=1}^{\infty}$ be proper, convex and lower semicontinuous functions on $E$. Suppose that
\[
B_3:=\bigcap_{k=1}^{\infty}\operatorname*{argmin}_{x\in E} h_k(x)
\]
is nonempty, closed and convex. For each $k\in\mathbb N$, let $M_k=\partial h_k$. Let $\rho:\mathbb N\to\mathbb N$ be an index map such that every positive integer occurs infinitely often. Define
\[
\left\{
\begin{array}{ll}
x_1=x\in E,\quad C_1=E,\\
q_n=(J+r_n\partial h_{\rho(n)})^{-1}Jx_n,\\
C_{n+1}=\{v\in C_n:\phi(v,q_n)\leq\phi(v,x_n)\},\\
x_{n+1}=\Pi_{C_{n+1}}x,
\end{array}
\right.
\]
where $\{r_n\}\subset[a,\infty)$ for some $a>0$. Then $x_n\to\Pi_{B_3}x$.
\end{corollary}

\begin{proof}
For each $k$, the subdifferential $\partial h_k$ is maximal monotone, and the Fermat rule for convex subdifferentials (see, for example, \cite[Section 23]{Rockafellar1970CA}) gives
\[
(\partial h_k)^{-1}0=\operatorname*{argmin}_{x\in E}h_k(x).
\]
Thus the assertion follows from Corollary \ref{maxonly} by taking $M_k=\partial h_k$ for every $k\in\mathbb N$.
\end{proof}

\begin{corollary}\label{hilbert}
Let $E=H$ be a real Hilbert space and let $C$ be a nonempty closed and convex subset of $H$. Let $\{f_i\}_{i=1}^{\infty}$ be a countable family of bifunctions from $C\times C$ into $\mathbb{R}$ satisfying $(A1)-(A4)$, let $\{A_j\}_{j=1}^{\infty}$ be a countable family of continuous monotone maps from $C$ into $H$, let $\{M_k\}_{k=1}^{\infty}$ be a countable family of maximal monotone operators on $H$, and let $\{T_n\}$ be a countable family of nonexpansive-type maps satisfying the NST-condition with $\Gamma$. Put
\[
F(\Gamma):=\bigcap_{T\in\Gamma}F(T),
\]
where $F(T)=\{z\in C:Tz=z\}$. Suppose
\[
B=F(\Gamma)\cap\left(\bigcap_{i=1}^{\infty}EP(f_i)\right)
\cap\left(\bigcap_{j=1}^{\infty}VI(C,A_j)\right)
\cap\left(\bigcap_{k=1}^{\infty}M_k^{-1}0\right)\neq\emptyset.
\]
Then the Hilbert space version of \eqref{alg}, with $J=I$ and $\Pi_{C_n}=P_{C_n}$, converges strongly to $P_Bx$.
\end{corollary}

\begin{proof}
In a Hilbert space the normalized duality mapping is the identity, the Lyapunov functional reduces to $\phi(x,y)=\|x-y\|^2$, and the generalized projection $\Pi_D$ coincides with the metric projection $P_D$ onto each nonempty closed and convex set $D$. Therefore all resolvent and shrinking-projection steps in \eqref{alg} reduce to their Hilbert-space counterparts. Applying Theorem \ref{mainthm} with $J=I$ gives the stated strong convergence to $P_Bx$.
\end{proof}

\begin{remark}
Theorem \ref{mainthm} and its corollaries are applicable in classical Banach spaces such as $L_p$, $\ell_p$ and $W_p^m(\Omega)$, $1<p<\infty$, under the usual smoothness and uniform convexity assumptions. In these spaces, the normalized duality map has known analytic representations, which makes the algorithm more explicit.
\end{remark}

\begin{proposition}[Finite truncations need not recover the countable problem]\label{finite-truncation-prop}
There exists a countable family of maximal monotone inclusion problems in a uniformly smooth and uniformly convex Banach space for which the full countable intersection is a singleton, while every finite truncation has an infinite-dimensional solution set.
\end{proposition}

\begin{proof}
Let $E=\ell_p$, $1<p<\infty$, and for each $k\in\mathbb N$ define
\[
h_k(x)=\frac{1}{p}|x_k|^p,\qquad x=(x_1,x_2,\ldots)\in\ell_p.
\]
Then $h_k$ is proper, convex and continuous. Define
\[
\psi_p(t)=
\begin{cases}
|t|^{p-2}t, & t\neq 0,\\
0, & t=0.
\end{cases}
\]
The subdifferential $M_k:=\partial h_k$ is the single-valued maximal monotone operator
\[
M_kx=\partial h_k(x)=\psi_p(x_k)e_k^{*},
\]
where $e_k^{*}$ is the $k$th coordinate functional in $\ell_q=(\ell_p)^{*}$, $1/p+1/q=1$. Hence
\[
M_k^{-1}0=(\partial h_k)^{-1}0=\{x\in\ell_p:x_k=0\}.
\]
Consequently,
\[
\bigcap_{k=1}^{\infty}M_k^{-1}0=\{0\},
\]
whereas, for each finite $N$,
\[
\bigcap_{k=1}^{N}M_k^{-1}0
=\{x\in\ell_p:x_1=x_2=\cdots=x_N=0\},
\]
which is an infinite-dimensional closed subspace of $\ell_p$. Thus no finite truncation captures the full countable intersection in this example.
\end{proof}

\begin{example}[An illustrative Hilbert-space specialization]\label{hilbert-illustrative-example}
Let $H=\ell_2$, let $C=H$, take $T_n=I$ for all $n$, and take the trivial equilibrium and variational inequality data
\[
f_i(u,v)=0,\qquad A_jx=0,
\]
for all $i,j\in\mathbb N$. For each $k\in\mathbb N$, set
\[
h_k(x)=\frac12 |x_k|^2,\qquad M_k=\partial h_k,
\]
so that $M_kx=x_ke_k$ and $M_k^{-1}0=\{x\in\ell_2:x_k=0\}$. Hence the common solution set is
\[
B=\bigcap_{k=1}^{\infty}M_k^{-1}0=\{0\}.
\]
If $\rho$ is the triangular index map described above before \eqref{Bset}, then every coordinate constraint is activated infinitely often. In this setting the maximal monotone resolvent is explicit:
\[
q_n=(I+r_nM_{\rho(n)})^{-1}x_n,
\]
that is,
\[
(q_n)_m=(x_n)_m \quad (m\neq \rho(n)),\qquad
(q_n)_{\rho(n)}=\frac{(x_n)_{\rho(n)}}{1+r_n}.
\]
Thus each activated step damps one coordinate of the current iterate, and the shrinking projection step keeps the iterates in nested half-spaces containing the common solution set $\{0\}$. Theorem \ref{mainthm} therefore yields $x_n\to 0$, the metric projection of the initial point onto $B$. This example illustrates the role of the index map in a concrete Hilbert-space case.
\end{example}

\begin{example}[Common minimizers of countably many convex functions]
Let $E=\ell_p$, $1<p<\infty$, and let $\{h_k\}_{k=1}^{\infty}$ be proper, convex and lower semicontinuous functions on $E$ such that
\[
\bigcap_{k=1}^{\infty}\operatorname*{argmin}_{x\in E}h_k(x)\neq\emptyset.
\]
For each $k\in\mathbb N$, set $M_k=\partial h_k$. Then each $M_k$ is maximal monotone and
\[
M_k^{-1}0=\operatorname*{argmin}_{x\in E}h_k(x).
\]
Corollary \ref{convexmin} therefore gives a strongly convergent proximal-type hybrid algorithm for finding the generalized projection of the initial point onto the common minimizer set of the countable family $\{h_k\}$.
\end{example}

\section{Conclusion}

In this paper, we introduced a hybrid scheme for approximating a common element of the solution sets of countable families of equilibrium problems, variational inequality problems and maximal monotone inclusion problems, together with the common $J$-fixed point set of a countable family of generalized $J_{*}$-nonexpansive mappings. The method uses equilibrium resolvents, variational inequality resolvents, generalized resolvents of maximal monotone operators and a shrinking projection technique. The strong convergence of the generated sequence was established in uniformly smooth and uniformly convex Banach spaces.

The theorem extends related hybrid methods in several directions. First, the equilibrium and variational inequality components are treated as countable families rather than finite families. Secondly, zeros of a countable family of maximal monotone operators are incorporated into the same framework. Thirdly, the residual convergence statement shows that the auxiliary resolvent steps asymptotically satisfy the activated component problems. Since maximal monotone operators include subdifferentials of proper convex lower semicontinuous functions, the result also covers common convex minimization problems. Proposition \ref{finite-truncation-prop} shows that the infinite-family setting cannot in general be recovered from a finite-family theorem by a direct truncation argument, while Example \ref{hilbert-illustrative-example} illustrates the algorithm in a concrete Hilbert-space case with explicit resolvent steps. The Hilbert-space and classical Banach-space consequences show that the theorem can be specialized to settings in which the resolvents and projections have more explicit forms.

\end{document}